\documentclass[11pt]{amsart}

\usepackage{geometry}
\usepackage{amsmath}
\usepackage{graphicx}
\usepackage{amssymb}
\usepackage{url}

\usepackage{todonotes}

\usepackage[all]{xy}
\xyoption{arc}

\newcommand{\st}{{\mbox{\scriptsize st}}}

\newcommand{\nth}{{\mbox{\scriptsize th}}}

\newcommand{\Q}{\mathbb Q}
\newcommand{\C}{\mathbb C}
\newcommand{\Z}{\mathbb Z}
\newcommand{\F}{\mathbb F}
\newcommand{\Zmod}[1]{\mathbb Z_{/#1}}
\newcommand{\Zp}{\Zmod{p}}

\DeclareMathOperator{\Aut}{Aut}
\DeclareMathOperator{\Hom}{Hom}
\DeclareMathOperator{\Res}{Res}
\DeclareMathOperator{\SL}{SL}
\DeclareMathOperator{\GL}{GL}
\DeclareMathOperator{\Qd}{Qd}
\DeclareMathOperator{\Ext}{Ext}

\newcommand{\legendre}[2]{\genfrac{(}{)}{}{}{#1}{#2}}

\newtheorem{theorem}{Theorem}[section]
\newtheorem{corollary}[theorem]{Corollary}
\newtheorem{proposition}[theorem]{Proposition}

\newtheorem{lemma}[theorem]{Lemma}
\theoremstyle{definition}
\newtheorem{definition}[theorem]{Definition}

\newtheorem{remark}[theorem]{Remark}

\usepackage{color}

\usepackage{tikz}
\usepackage{tikz-cd}
\usetikzlibrary{cd}

\begin{document}
\title{$\Zp \times \Zp$ actions on $S^n \times S^n$}

\author{J. Fowler}
\address{Department of Mathematics, The Ohio State University, Columbus, Ohio, USA 43210}
\email{fowler@math.osu.edu}

\author{C. Thatcher}
\address{Department of Mathematics and Computer Science, University of Puget Sound, Tacoma, Washington, USA 98416}
\email{cthatcher@pugetsound.edu}


\date{}

\maketitle

\begin{abstract}
  We determine the homotopy type of quotients of $S^n \times S^n$ by free
  actions of $\Zp \times \Zp$ where $2p>n+3$. Much like free $\Zp$ actions, they can be 
  classified via the first $p$-localized $k$-invariant, but there are restrictions on the
  possibilities, and these restrictions are sufficient to determine every possibility in
  the $n=3$ case. We use this to complete the classification of free $\Zp \times \Zp$ 
  actions on $S^3 \times S^3$, for $p>3$, by reducing the problem to the simultaneous 
  classification 
  of pairs of binary quadratic forms. Although the restrictions are not sufficient to 
  determine which $k$-invariants are realizable in general, they can sometimes be used
  to rule out free actions by groups that contain $\Zp\times\Zp$ as a normal Abelian subgroup. 
 
\end{abstract}

\section{Introduction}
The topological spherical space form problem asks: what groups can act freely on the sphere and how can these group actions be classified? Conditions for which groups can act were determined during the middle of the last century (see e.g. \cite{MR0010278} \cite{MR90056}  \cite{MR0426003}). The question of {\it{how}} free cyclic groups can act on spheres was addressed in the study of lens spaces with the classification of all free cyclic group actions being completed somewhat recently. 
This question can easily be extended to actions on products of spheres. What groups can act has been addressed in a number of papers (see e.g. \cite{MR0096235} \cite{MR0107866} \cite{MR561237} \cite{MR1745517} \cite{MR892053} \cite{MR2592957} \cite{okay2018}), while the classification of {\it{how}} the simplest of groups do act on products of spheres and what invariants distinguish them has largely been skipped. This paper focuses specifically on the {\it{how}} question.

To begin addressing how groups act, one might consider the simplest group actions. Free $\Z/p$ actions on $S^n\times S^n$, for $p>\frac{1}{2}(n+3)$, was addressed in \cite{MR2721633} - the homotopy type is determined completely by the homotopy groups and the first $k$-invariant. In this paper, we consider quotients 
of free actions of $\Zp \times \Zp$ on $S^n \times S^n$ with $n>1$ odd and $p>\frac{1}{2}(n+3)$. It turns out that the homotopy classification 
is similar to the $\Zp$ case - the classes are determined by the first $k$-invariants, but the $k$-invariants are more complicated.  A significant insight is the usefulness of localizing at a large prime -  while the homotopy groups of spheres are replete with torsion,
$\pi_i S^n$ has no $p$-torsion for $i\leq 2n$ when $p$ is reasonably large. From this we see that only a couple of nontrivial stages in the localized Postnikov tower carry all the relevant data for our study.

In this paper, we begin with a review of the cohomology of $\Zp \times \Zp$ in Section~\ref{section:cohomology} and then proceed with the classification from there. In Section~\ref{section:homotopy-equivalence}, we determine the
homotopy type in terms of a single $k$-invariant, or equivalently, in
terms of the transgression in a certain spectral sequence, which the reader might also prefer to think of as an Euler class.  The 
homotopy classification of $\Zp \times \Zp$
actions on $S^n \times S^n$ then amounts to a choice of 
parameters in $\Zp$.

In Section~\ref{section:k-invariants} we find that there are strong
restrictions on the possible $k$-invariants. In Section~\ref{section:constructions} we 
provide constructions of the possible homotopy classes based on these restrictions, and in 
Section~\ref{section:s3-times-s3} we show that this is the full homotopy classification of $\Zp\times\Zp$
actions on $S^3 \times S^3$ by reducing the classification to that of
pairs of binary quadratic forms. One of our main results is the following.
\newtheorem*{thm:summary}{Theorem \ref{thm:summary}}
\begin{thm:summary} Let $p > 3$ be prime. If $p \equiv 1 \bmod 4$, then there are four homotopy classes of quotients of $S^3 \times S^3$ by free $\Zp \times \Zp$ actions.  If $p \equiv 3 \bmod 4$, then there are two classes.
\end{thm:summary}

Finally, in Section~\ref{section:NormalAbelianSubgroup} we show that these restrictions
can be
used to rule out free actions by groups containing $\Zp\times\Zp$ as a
normal Abelian subgroup. This is consistent with the results about
$\Qd(p)$ in a recent paper by Okay--Yal{\c{c}}in \cite{okay2018}.

We note that a subsequent paper will provide the
homeomorphism classification of these quotients in the case of linear
actions.

\subsection*{Acknowledgements}
The authors thank Ian Hambleton for helpful conversations and an anonymous reviewer for useful suggestions.

\section{The cohomology of $\Zp\times \Zp$}
\label{section:cohomology}

To begin, we will need the integral cohomology of $X=(S^n\times S^n)/(\Zp\times \Zp)$. To determine this we first
need to consider the ring structure of the integral cohomology of $\Zp\times\Zp$. It is known that $H^*(\Zp;\Zp)=\F_p[a]\otimes\wedge(u)$, where $|u|=1$, $|a|=2$, and $\beta(u)=a$ with $\beta$ the Bockstein homomorphism, and  that $H^*(\Zp;\Z)=\Z[a]/(pa)$, where $|a|=2$. It follows from the K\"unneth Theorem that $H^*(\Zp\times\Zp;\Zp)\cong \F_p[a,b]\otimes\wedge(u,v)$, where $|u|=|v|=1$ and $|a|=|b|=2$, but $H^*(\Zp\times\Zp;\Z)$ requires a bit more work.

The homology and cohomology groups themselves can be determined using the K\"unneth Theorem and universal coefficients.

\begin{proposition}
The integral homology groups of $\Zp\times \Zp$ are:
\[
H_k(\Zp\times\Zp\;;\;\Z) \cong \left\{
\begin{array}{ll}
\Z & \text{for } k = 0,\\
(\Zp)^{(k+3)/2} & \text{for } k > 0 \text{ odd},\\
(\Zp)^{k/2} & \text{for } k >0 \text{ even}.
\end{array} \right.
\]
The integral cohomology groups are:
\[
H^k(\Zp\times\Zp\;;\;\Z) \cong \left\{
\begin{array}{ll}
\Z & \text{for } k = 0,\\
0 & \text{for } k=1,\\
(\Zp)^{(k-1)/2} & \text{for } k > 1 \text{ odd},\\
(\Zp)^{(k+2)/2} & \text{for } k > 1 \text{ even}.
\end{array} \right.
\]
\end{proposition}

The ring structure can then be determined by piecing together the exact sequences in cohomology associated to the short exact sequences $0\to \Zp\to\Zp^2\to\Zp\to 0$ and $0\to\Z\to\Z\to\Zp\to 0$. We take $G=\Zp\times\Zp$ for the sake of ease in writing diagrams and throughout the paper.

\[\xymatrix{H^n(G;\Z)\ar[r]^{\rho} & H^n(G;\Zp)\ar[r]^{\tilde{\beta}}\ar[rd]_{\beta} & H^{n+1}(G;\Z)\ar[r]^p\ar[d]^{\rho} & H^{n+1}(G;\Z)\\
 & & H^{n+1}(G;\Zp) & }\]

Here $\beta$ is the Bockstein associated to the first short exact
sequence above, $\tilde{\beta}$ is the Bockstein associated to the
second one, $\rho$ is the homomorphism induced by the map $\Z\to\Zp$, and $p$ is the map induced
by multiplication by $p$. We note that the triangle commutes. This along with the ring structure of $H^*(\Zp\times \Zp;\Zp)$ allows one to find the ring structure of
$H^*(\Zp\times \Zp;\Z)$.  This ring structure is given, among other
places, in \cite{MR675422} and \cite{Roberts2013}.

\begin{theorem}\label{BGintcohlgy}
The integral cohomology ring of $\Zp\times\Zp$ is
\[
  H^*(\Zp\times\Zp;\Z)\cong \Z[a,b,c]/(pa,pb,pc,c^2) 
\]
where $|a|=|b|=2$ and $|c|=3$.
\end{theorem}

\section{Homotopy equivalence and the $k$-invariants}
\label{section:homotopy-equivalence}

Let $G=\Zp\times\Zp$ act freely on $S^n\times S^n$, and let $X$ be the
resulting quotient manifold, which may only be a TOP manifold.  A simple example of such an action is given by the first $\Zp$ acting freely on the first $S^n$ and the second $\Zp$ acting freely on the second $S^n$ in such a way that the resulting quotient manifold is the product of two lens spaces. We wish to determine when two arbitrary free actions of $\Zp\times\Zp$ result in homotopy equivalent quotients. 

For $p>3$, the fundamental group $\pi_1(X)=G$ acts trivially on the homology of the universal cover of $X$ because $\GL_2(\Z)$ has no $p$-torsion.  It follows that $X$ is nilpotent, hence $X$ has a Postnikov tower that admits principal refinements and $X$ can be $p$-localized.

\begin{definition}
A connected space $X$ \textit{$n$-simple} if $\pi_1(X)$ is Abelian and acts trivially on $\pi_i(X)$ for $1<i\leq n$.
\end{definition}

An $n$-simple space has a Postnikov tower that consists of principal fibrations through the $n$th stage. We briefly describe the construction, but more specific details can be found in \cite{May}.  The first stage is taken to be $X_1=K(\pi_1(X),1)$, with $f_1:X\to X_1$ inducing an isomorphism on $\pi_1$. $p_i:X_i\to X_{i-1}$ is constructed iteratively as the fibration induced from the path space fibration over $K(\pi_iX,i+1)$ by the map $k^{i+1}:X_{i-1}\to K(\pi_iX, i+1)$. The $k^{i+1}$ are called $k$-invariants, and are thought of as cohomology classes. There are maps $f_i:X\to X_i$, for $1\leq i\leq n$, such that $p_i\circ f_i=f_{i-1}$, and each $f_i$ induces an isomorphism on $\pi_k$ for all $k\leq i$. Additionally, $\pi_k(X_i)=0$ for $k>i$.

The bottom of the Postnikov tower for an $n$-simple space generically looks like:

\begin{displaymath}
\xymatrix{
& \ar[d] & \\
& X_{n} \ar[d]& \\
&\vdots \ar[d] & \\
 & X_{3}\ar[d]^{p_3}\ar[r]^-{k^{5}} & K(\pi_{4}(X),5) \\
 & X_2\ar[d]^{p_2}\ar[r]^-{k^{4}} & K(\pi_{3}(X),4) \\
X\ar[r]\ar[ru]^{f_2}\ar[ruu]^{f_3}\ar[ruuuu]^{f_n} & X_1= K(G,1)\ar[r]^-{k^3} & K(\pi_2(X),3)}
\end{displaymath}

\begin{lemma}
\label{lemma:simple}
Let $n\ge3$. For $p>3$, $X=(S^n\times S^n)/(\Zp\times\Zp)$ is $n$-simple.
\end{lemma}

\begin{proof}
Since $\pi_{i}(X)\cong\pi_i(S^n\times S^n)\cong \pi_i(S^n)\times \pi_i(S^n)$, we see that $\pi_2(X)\cong\pi_3(X)\cong\dots\cong\pi_{n-1}(X)=0$, hence there is one nontrivial homotopy group $\pi_iX$ for $1<i<n+1$:  $\pi_n(X)\cong\pi_n(S^n\times S^n) =\Z^2$. Since $\Aut(\Z^2)$ only has $2$-torsion and $3$-torsion, and $p>3$, $\pi_1$ acts trivially on $\pi_i(X)$ for $1<i\leq n$.
\end{proof}

Since $\pi_i(X)$ is trivial for $1<i<n$, $X_1\simeq X_2\simeq\dots\simeq X_{n-1}$, and the bottom of the Postnikov tower becomes

\begin{displaymath}
\xymatrix{
& \ar[d] & \\
 & X_n\ar[d]^{p_n} & \\    X\ar[r]\ar[ru]^{f_n} & X_1= K(G,1)\ar[r]^-{k^{n+1}} & K((\Z)^2,n+1)}
\end{displaymath}

As $X$ is nilpotent, the Postnikov tower above the $n$th step admits principal refinements. Specifically, following the notation in \cite{MR2884233}, there is a central $\pi_1(X)$-series, $1=G_{j,r_j}\subset\cdots\subset G_{j,0}=\pi_j(X)$ for each $j>n$, such that $A_{j,l}=G_{j,l}/G_{j,l+1}$ for $0\leq l<r_j$ is Abelian and $\pi_1(X)$ acts trivially on $A_{j,l}$. The $n+1$st stage is then a finite collection of spaces $X_{n+1,l}$ constructed from maps $k^{n+2,l}:X_{n+1,l}\to K(A_{n+1,l},n+2)$ and with $X_{n+1,0}=X_n$. Similarly, the $n+i$th stage is a finite collection of spaces $X_{n+i+1,l}$ constructed from maps $k^{n+i+1,l}:X_{n+i,l}\to K(A_{n+i,l},n+i+1)$ and with $X_{n+i,0}=X_{n+i-1,r_{n+i-1}}$.

Additionally, since $X$ is nilpotent, $X$ can be $p$-localized. This is done by inductively $p$-localizing the Postnikov tower, i.e. the $(X_j)_{(p)}$ are inductively constructed using fibrations with $K(\pi ,j)'s$, where each $\pi$ is a $\Z_{(p)}-$module (see for example Theorem 5.3.2 in \cite{MR2884233} or Sullivan's notes \cite{MR2162361}). 
Specifically, we localize the first stage, $(X_1)_{(p)}=(K(\pi_1(X),1))_{(p)}=K((\pi_1(X))_{(p)},1)=K((\Zp)^2,1)=X_1$, and localize the $n$th homotopy group, $(\pi_nX)_{(p)}=\pi_n X\otimes \Z_{(p)} =(\Z_{(p)})^2$, and then consider the following diagram:

\begin{displaymath}
\xymatrix{
K(\pi_nX,n)\ar[r]\ar[d] & X_n \ar[r]\ar@{-->}[d]^{\phi_{n+1}} & X_1\ar[r]^(.3){k^{n+1}}\ar[d]^{\phi_n} & K(\pi_nX,n+1))\ar[d]\\
K((\pi_nX)_{(p)},n)\ar[r] & (X_n)_{(p)} \ar[r] & (X_1)_{(p)}\ar[r]^(.3){(k^{n+1})_{(p)}} & K((\pi_nX)_{(p)},n+1))}
\end{displaymath}

\noindent Here $(k^{n+1})_{(p)}$ is the $p$-localization of $k^{n+1}$. The right square commutes up to homotopy and there exists a map $\phi_{n+1}$, that is localization of $X_n$ at $p$, such that the middle and left squares commute up to homotopy. Similar arguments can be made for the stages above $n$, and then we take $X_{(p)}=\lim (X_i)_{(p)}$ and $\phi=\lim \phi_i:X\to X_{(p)}$. 

We note that the unique map (up to homotopy) $\phi$ localizes the homotopy and homology groups of $X$. In particular, $\phi_*(\pi_iX)=(\pi_iX)_{(p)}$, and further, $\phi_*:[X_{(p)},Z]\to[X,Z]$ is an isomorphism for any $p$-local space $Z$. \cite{MR0478146} \cite{MR2884233}.

By \cite{MR693421} the unstable homotopy group $\pi_i(S^n)$ has no $p$-torsion for $i<2p+n-3$.  We restrict to $p>(n+3)/2$, so that $\pi_i(X)$ has no $p$-torsion for $i\leq 2n$. It follows that $A_{j,l}$ has no $p$-torsion for $n<j\leq 2n$, and since $A_{j,l}$ is finite for all $j>n$, $(A_{j,l})_{(p)}=A_{j,l}\otimes Z_{(p)}=0$, $K((A_{j,l})_{(p)},j+1)$ is a point, and $(k^{m+1})_{(p)}=0$, where $(k^{m+1})_{(p)}$ is the $p$-localized $k$-invariant associated with the $m$th stage ($X_m=X_{j,l}$). 
Since the construction of the tower becomes formal after the dimension of $X$ (i.e after $2n$), the only nontrivial $k$-invariant in the localized Postnikov tower before it becomes formal is $(k^{n+1})_{(p)}\in H^{n+1}(X_1;(\Z_{(p)})^2)\cong (\Zp)^{n+3}$. Given an identification of $\pi_1$ and $\pi_n$, this $p$-localized first $k$-invariant then determines the homotopy type of the localization.  In fact, the first nontrivial $k$-invariant characterizes $X$ up to homotopy as well.

\begin{theorem}\label{thm:he2}
  Let $X$ and $Y$ be quotients of free $\Zp\times\Zp$ actions on
  $S^n\times S^n$ with odd $n \geq 3$, where $p > 3$ satisfies $2p+n-3>2n$, and let $k_X^{n+1}$ and  $k_Y^{n+1}$ denote the first nontrivial $k$-invariant. The spaces $X$ and $Y$ are
  homotopy equivalent if and only if there are isomorphisms  $g_1 : \pi_1X \to \pi_1Y$ and $g_n : \pi_nX \to \pi_nY$ so that
  \begin{displaymath}
\xymatrix{
K(\pi_1(X) ,1)\ar[r]^-{k_X^{n+1}}\ar@{->}[d]^{{g_1}_\star} &
K(\pi_n(X),n+1)\ar@{->}[d]^{{g_n}_\star}\\
K(\pi_1(Y) ,1)\ar[r]^-{k_Y^{n+1}} &
K(\pi_n(Y),n+1)
}
\end{displaymath}
  commutes up to homotopy, i.e., $k_X^{n+1}\in H^{n+1}(\pi_1 X;\pi_n X)$ and $k_Y^{n+1}\in H^{n+1}(\pi_1 Y;\pi_n Y)$ are identified through 
the maps induced by $g_1$ and $g_n$.
\end{theorem}

Two lemmas are used in proving
Theorem~\ref{thm:he2}. Lemmas~\ref{lemma:postnikov-equivalence} and~\ref{lemma:two} are related to
Lemmas~1 and~2 in \cite{MR2721633}, respectively.

\begin{lemma}
  \label{lemma:postnikov-equivalence}
Let $X$ and $Y$ be $n$-simple spaces with identifications $\pi_1(X)\cong\pi_1(Y)\cong G$ and $\pi_n(X)\cong\pi_n(Y)\cong H$.  Further suppose $\pi_i(X)=\pi_i(Y)=0$ for $1<i<n$. If, as in the statement of Theorem~\ref{thm:he2}, the identifications on $\pi_1$ and $\pi_n$ provide an identification of the first nontrivial $k$-invariants of $X$ and $Y$ in $H^{n+1}(G;H)$, then the $n^\nth$ stages of the Postnikov towers for $X$ and $Y$ are 
homotopy equivalent, i.e. $X_n\simeq Y_n$.
\end{lemma}

\begin{proof}
We have isomorphisms 
$g_1 : \pi_1X \to \pi_1Y$ and $g_n : \pi_nX \to \pi_nY$, and
$k_X^{n+1}$ and $k_Y^{n+1}$ are the first nontrivial $k$-invariants of $X$
and $Y$, respectively.  The $k$-invariant is regarded as a map \[
k_X^{n+1} : K(\pi_1(X),1) \to K(\pi_n(X),n+1).\]
The isomorphism $g_1$ induces a homotopy equivalence ${g_1}_\star : K(\pi_1(X),1) \to K(\pi_1(Y),1)$.
Similarly, the isomorphism $g_n$ induces a homotopy equivalence ${g_n}_\star : K(\pi_n(X),n+1) \to K(\pi_n(Y),n+1)$.  The identification of the first nontrivial $k$-invariant means that ${g_n}_\star \circ k_X^{n+1}$ is homotopic to $k_Y^{n+1} \circ {g_1}_\star$.

The $n^\nth$ stage $X_n$ of the Postnikov tower is constructed as the
pullback of the pathspace fibration over $K(\pi_n(X),n+1)$ and $k_X^{n+1}$:

\begin{displaymath}
\xymatrix{
X_{n}\ar[r]\ar[d] &
(K(\pi_n(X),n+1))^I\ar[d]\\
K(\pi_1(X) ,1)\ar[r]^-{k_X^{n+1}} & K(\pi_1(X),n+1)}
\end{displaymath}

\noindent A similar construction is performed for $Y_n$.   We have the following map of fibrations, and we want to define a map $f$ on
the fibers.
\begin{displaymath}
\xymatrix{
X_n\ar[r]\ar@{-->}[d]^f & K(\pi_1(X) ,1)\ar[r]^-{k_X^{n+1}}\ar@{->}[d]^{{g_1}_\star} &
K(\pi_n(X),n+1)\ar@{->}[d]^{{g_n}_\star}\\
Y_n\ar[r] & K(\pi_1(Y) ,1)\ar[r]^-{k_Y^{n+1}} &
K(\pi_n(Y),n+1)
}
\end{displaymath}

\noindent The identification of the first nontrivial $k$-invariants means the square on the right commutes up to homotopy.
Let $h:K(\pi_1(X) ,1)\times I\to K(\pi_n(Y),n+1)$ be a homotopy from
${g_n}_\star \circ k_X^{n+1}$ to $k_Y^{n+1} \circ {g_1}_\star$.  With $X_n$ defined as a pullback, a point in $X_n$ consists of a pair $(x,q)$ with $x \in K(\pi_1(X),1)$ and 
$q:I\to K(\pi_n(X),n+1)$ satisfying $q(1)=k_X^{n+1}(x)$.  Define $f : X_n \to Y_n$ by $f(x,q) = (y,r)$ with $y = {g_1}_\star (x)$ and $r : I\to K(\pi_n(Y),n+1)$ given by
\[
r(t) = \begin{cases}
{g_n}_\star q(t) & \mbox{if $t \leq 1/2$,} \\
h(x, 2t-1) & \mbox{if $t \geq 1/2$.}
\end{cases}
\]
This provides a construction for $f : X_n \to Y_n$,
and by a
theorem of Milnor's, the fibers $X_n$ and $Y_n$ have the
homotopy types of CW complexes.  Therefore we have a commuting diagram of
homotopy groups,
\begin{displaymath}
\xymatrix{
\scriptstyle{\pi_{j+1}K(\pi_1(X),1)}\ar[r]\ar[d]^{\cong} &
\scriptstyle{\pi_{j+1}K(\pi_n(X),n+1)}\ar[r]\ar[d]^{\cong} &
\scriptstyle{\pi_j X_n}\ar[r]\ar[d] &
\scriptstyle{\pi_{j}K(\pi_1(X),1)}\ar[r]\ar[d]^{\cong} &
\scriptstyle{\pi_{j}K(\pi_n(X),n+1)}\ar[d]^{\cong}\\
\scriptstyle{\pi_{j+1}K(\pi_1(Y),1)}\ar[r] &
\scriptstyle{\pi_{j+1}K(\pi_n(Y),n+1)}\ar[r] &
\scriptstyle{\pi_j Y_n}\ar[r] &
\scriptstyle{\pi_{j}K(\pi_1(Y),1)}\ar[r] &
\scriptstyle{\pi_{j}K(\pi_n(Y),n+1)}
}
\end{displaymath}
The five-lemma gives us that $\pi_jX_n\cong\pi_jY_n$ for
all $j$.  Thus we have a weak equivalence between spaces having the homotopy type of CW complexes, so we have a homotopy equivalence.
\end{proof}

\begin{lemma}\label{lemma:two}
Let $M$ and $N$ be nilpotent spaces such that $H^n(M;\Z) = 0$ and $H^n(N;\Z) = 0$ for $n>m$, for some $m>0$.  If the $m$th stage of the Postnikov
tower for $M$ is homotopy equivalent to the $m$th stage of the Postnikov tower for $N$, then $M$ is homotopy equivalent to $N$, i.e. if $M_m\simeq N_m$ then $M\simeq N$.
\end{lemma}

We note that this lemma is essentially Lemma 2 in \cite{MR2721633} (the difference being the change of ``$m$-dimensional'' to the cohomology requirement above), and the obstruction argument proof works exactly as written. We are now in a position to prove Theorem \ref{thm:he2}.

\begin{proof}[Proof of Theorem \ref{thm:he2}]
As has been our convention, let $G=\Zp\times\Zp$.

In one direction, we assume there is a homotopy equivalence from $X$ to $Y$.  On $\pi_1$, the homotopy equivalence provides an isomorphism which then yields a homotopy equivalence between the first stage of a Postnikov tower of $X$ and the same of $Y$.  The next nontrivial stage is stage $n$, and we have a commutative square
\begin{displaymath}
\xymatrix{
X_n\ar[r]\ar[d] & Y_n \ar[d] \\
X_1\ar[r] & Y_1
}
\end{displaymath}
The vertical maps are fibrations, and taking the cofibers of these vertical maps yields the commutative square displayed in the statement of Theorem \ref{thm:he2}.

To prove the other direction, we assume $\pi_1X$ and $\pi_1Y$ are identified with $G$ and that this gives an isomorphism $g_1:\pi_1X\to\pi_1Y$, and $\pi_nX$ and $\pi_nY$ are identified with $(\Z)^2$ and that this gives an isomorphism $g_n:\pi_nX\to\pi_nY$. These maps induce identifications of $\pi_1X_{(p)}\cong\pi_1Y_{(p)}\cong G\otimes\Z_{(p)}\cong G$ and $\pi_nX_{(p)}\cong\pi_nY_{(p)}\cong(\Z)^2\otimes\Z_{(p)}\cong(\Z_{(p)})^2$ after localizing the Postnikov systems of both $X$ and $Y$ at $p$.
We have that $(X_1)_{(p)}=K(\pi_1X_{(p)},1)\simeq (Y_1)_{(p)}= K(\pi_1Y_{(p)},1)$.
Let $k_X^{n+1}$ and $k_Y^{n+1}$ be the first nontrivial $k$-invariants of $X$ and $Y$, respectively, and take $(k_X^{n+1})_{(p)}$ and $(k_Y^{n+1})_{(p)}$ to be the $p$-localized first $k$-invariants, respectively.
Since $k_X^{n+1}$ and $k_Y^{n+1}$ are in the same homotopy class of maps in
$[K(G,1):K((\Z)^2 ,n+1)]$, $(k_X^{n+1})_{(p)}$ and $(k_Y^{n+1})_{(p)}$ will be in the
same homotopy class of maps in $[K(G,1):K((\Z_{(p)})^2,n+1)]$ by construction. Since $p>3$, $X$ and $Y$ are $n$-simple by Lemma~\ref{lemma:simple}. Given that the localization of both spaces and their homotopy groups preserves this property, we see that $X_{(p)}$ and $Y_{(p)}$ are both $n$-simple as well, and we can apply Lemma~\ref{lemma:postnikov-equivalence}. It follows that $(X_n)_{(p)}\simeq (Y_n)_{(p)}$.

Since we are assuming $2p+n-3>2n$, we have that $(X_{2n+1,0})_{(p)}\simeq (X_n)_{(p)}\simeq (Y_n)_{(p)}\simeq (Y_{2n+1,0})_{(p)}$. It follows from 
Lemma \ref{lemma:two} that $X_{(p)}\simeq Y_{(p)}$. The maps $l_1:X_{(p)}\to X_{(0)}$ and $l_2:Y_{(p)}\to Y_{(0)}$, given by
inverting $p$, give via the naturality of localization a homotopy equivalence
$\iota :X_{(0)}\stackrel{\simeq}{\to}Y_{(0)}$ and identifications of $\pi_n X_{(0)}$ and $\pi_n Y_{(0)}$ with $(\Q)^2$. The following commutes up to homotopy:
\begin{displaymath}
\xymatrix{
X_{(p)}\ar[r]^{\simeq}\ar[d]^{l_1} &
Y_{(p)}\ar[d]^{l_2}\\
X_{(0)}\ar[r]^-{\iota} & Y_{(0)}}
\end{displaymath}

On the other hand, we can consider localization away from $p$. For $X$ we have the following commutative diagram:
\[
\xymatrix{
S^n\times S^n\ar[r]\ar[d]_q & (S^n\times S^n)[\frac{1}{p}]\ar[d]^{q[\frac{1}{p}]}\\
X \ar[r] & X[\frac{1}{p}]}
\]
Since $\pi_1(X[\frac{1}{p}])=G\otimes \Z [\frac{1}{p}]=0$, we see that
$\pi_j(X[\frac{1}{p}])\cong\pi_j((S^n\times S^n)[\frac{1}{p}])$ for all $j$.
Thus $q[\frac{1}{p}]$ induces an isomorphism on every homotopy group, and is a
homotopy equivalence since $(S^n\times S^n)[\frac{1}{p}]$ and $X[\frac{1}{p}]$
both have the homotopy types of CW complexes. Similarly we have a homotopy equivalence
$(S^n\times S^n)[\frac{1}{p}]\simeq Y[\frac{1}{p}]$.

By appropriately choosing a self-map of $S^n \times S^n$ and localizing, we can produce a self-map of $(S^n \times S^n)[1/p]$ acting as desired on $\pi_n$ and then can compose to produce a map
\[
X[1/p] \simeq (S^n \times S^n)[1/p] \to (S^n \times S^n)[1/p] \simeq Y[1/p].
\]
Together these maps 
give us a homotopy equivalence $X[\frac{1}{p}]\simeq Y[\frac{1}{p}]$.  Since we have maps
$X[\frac{1}{p}]\to X_{(0)}$ and $Y[\frac{1}{p}]\to Y_{(0)}$ given by inverting
everything else, the naturality of localization gives us a map
$\iota ':X_{(0)}\stackrel{\simeq}{\to}Y_{(0)}$.  It is a homotopy
equivalence because it induces an isomorphism on all of the homotopy groups.  We have a diagram that commutes up to homotopy,
\begin{displaymath}
\xymatrix{
X[\frac{1}{p}]\ar[r]^{\simeq}\ar[d]^{L_1} &
Y[\frac{1}{p}]\ar[d]^{L_2}\\
X_{(0)}\ar[r]^-{\iota'} & Y_{(0)}}
\end{displaymath}
Since $X_{(0)}$ and $Y_{(0)}$ are $K((\Q)^2,n)$'s,
homotopy classes of maps from $X_{(0)}$ to $Y_{(0)}$ are identified with elements of $\Hom(\pi_n X \otimes \Q ,\pi_n Y \otimes \Q)$, but by construction, $\iota$ and $\iota'$ are identified by their action on $\pi_n$.

The space $X$ is the homotopy pullback of 
$X_{(p)}$ and $X[\frac{1}{p}]$ along $l_1$ and $L_1$.
For $x \in X$, write $x_1$ and $x_2$ for the image of $x$ in  $X_{(p)}$ and $X[\frac{1}{p}]$, respectively, so $l_1(x_1)$ and $L_1(x_2)$ are connected by a path in $X_{(0)}$.  To map into $Y$, a homotopy pullback, it is enough to provide maps $X \to Y_{(p)}$ and $X \to Y[\frac{1}{p}]$ which agree up to a path in $Y_{(0)}$.  Combining the localization squares for $X$ and $Y$ and all of the maps we have constructed between the squares, we have the following cube that commutes up to homotopy, thereby providing maps $X \to Y_{(p)}$ and $X \to Y[\frac{1}{p}]$ which agree up to homotopy.
\[
\xy
(0,10)*{X[\frac{1}{p}]};(20,10)*{X_{(0)}};(0,30)*{X};(20,30)*{X_{(p)}};
{\ar (0,27);(0,13)};{\ar (3,30);(15,30)};{\ar^{l_1} (20,27);(20,13)};{\ar^{L_1} (4,10);(15,10)};
(10,0)*{Y[\frac{1}{p}]};(30,0)*{Y_{(0)}};(10,20)*{Y};(30,20)*{Y_{(p)}};
{\ar (10,17);(10,3)};{\ar (13,20);(25,20)};{\ar^{l_2} (30,17);(30,3)};{\ar^{L_2} (14,0);(25,0)};
{\ar@{-->}(3,27);(7,23)};{\ar@{-->}(23,7);(27,3)};{\ar^{\simeq} (23,27);(27,23)};{\ar^{\simeq} (3,7);(7,3)};
\endxy
\]

\noindent From this we obtain maps of short exact sequences on homotopy for all $j$.

\[
\xymatrix{
0\ar[r]\ar[d]^= & \pi_jX\ar[r]\ar[d] & \pi_jX_{(p)}\oplus\pi_jX[\frac{1}{p}]\ar[r]\ar[d]^{\cong} & 
\pi_jX_{(0)}\ar[r]\ar[d]^{\cong} & 0\ar[d]^=\\
0\ar[r]& \pi_jY\ar[r] & \pi_jY_{(p)}\oplus\pi_jY[\frac{1}{p}]\ar[r] & \pi_jY_{(0)}\ar[r] & 0}
\]

\noindent The five-lemma gives isomorphisms on the homotopy groups of $X$ and
$Y$.  This then gives a homotopy equivalence from $X$ to $Y$ as they are
 both CW complexes.

\end{proof}

\section{Restrictions on the first $k$-invariant}
\label{section:k-invariants}

Throughout this section we will continue to let $G=\Zp\times \Zp$ and $X:=(S^n\times S^n)/G$, where $p > 3$ is an odd prime  and $n\ge 3$ is odd.
The first stage of the Postnikov system provides a fibration: $K(\pi_n(X),n)\overset{j}\to X_n\to X_1= K(\pi_1X,1)$.  The space $X_n$ is induced from the path-space fibration over $K(\pi_n(X),n+1)$, so the fundamental group $\pi_1(X_1)=G$ acts trivially on the homology of $K(\pi_n(X),n)$. This results in an exact sequence
\[
\cdots\to H^n(X_n;\pi_n(X))\overset{j^*}\to H^n(K(\pi_n(X),n);\pi_n(X))\overset{\tau}\to H^{n+1}(X_1;\pi_n(X)),
\]
where $\tau$ is the transgression. By \textsection 6.2 of \cite{MR1793722}, the transgression $\tau$ is also the  differential $\tau=d_{n+1}:E_{n+1}^{0,n} \to E_{n+1}^{n+1,0}$ in the Serre spectral sequence of the fibration. 
As described in Chapter III Section 3.7 of \cite{MR1434104}, the fundamental classes of the fiber $K(\pi_n(X),n)$ and the base $X_1$ correspond 
under the transgression. If $\iota\in H^n(K(\pi_n(X),n);\pi_n(X))$ is the fundamental class of the fiber, the $k$-invariant, $k^{n+1}\in H^{n+1}(X_1;\pi_n(X))$, is the pullback of the fundamental class of the base space, and $\tau(\iota)=k^{n+1}$.

On the other hand, consider the Borel fibration:

\[S^n\times  S^n \overset{i}\to (S^n\times S^n)_{hG}\to BG,\]

\noindent where $(S^n\times S^n)_{hG} :=(EG\times S^n\times S^n)/G\simeq (S^n\times S^n)/G=X$. 

There is a map of fibrations:
\[
\begin{tikzcd}
S^n\times S^n \arrow{r}{i}\arrow{d}{\phi_n}  & X \arrow{r}{f_1} \arrow{d}{f_n}  & BG \arrow{d}{=}\\
K(\pi_n(X),n) \arrow{r}{j}  & X_n \arrow{r}{p_n} & BG
\end{tikzcd}
\]
where the map $\phi_n : S^n \times S^n \to K(\pi_n(X),n)$ classifies the fundamental class in $H^n(S^n \times S^n ; \Z^2)$, and $f_n:X \to X_n$ is the $n$-equivalence in the Postnikov tower. Since $\pi_1(BG)=G$ is a finite group generated by odd order elements, it acts trivially on the cohomology of the fiber (see \cite{okay2018}), and we obtain maps between the induced exact sequences in cohomology:

\[
\begin{tikzcd}[column sep=small]
\cdots \arrow{r} & H^n(X;\pi_n(X)) \arrow{r}{i^*}  & H^n(S^n\times S^n;\pi_n(X)) \arrow{r}{\bar{\tau}} & H^{n+1}(X_1;\pi_n(X)) \\
\cdots \arrow{r} & H^n(X_n;\pi_n(X)) \arrow{r}{j^*}\arrow{u}{f_n^*}  & H^n(K(\pi_n(X),n);\pi_n(X)) \arrow{r}{\tau}\arrow{u}{\phi_n^*} & H^{n+1}(X_1;\pi_n(X)) \arrow{u}{=}
\end{tikzcd}
\]

It follows that for the fundamental class $\iota\in H^n(K(\pi_n(X),n),\pi_n(X))$, which corresponds to the identity map under the equivalence  $H^n(K(\pi_n(X),n),\pi_n(X)) \cong \Hom(\pi_n(X),\pi_n(X))$ (from the Universal Coefficient Theorem), $\bar{\tau}(\phi_n^*(\iota))=\tau(\iota)=k^{n+1}$.  Further, since $H_{n-1}(S^n\times S^n)=0$, we also have from the Universal Coefficient Theorem, $H^n(S^n\times S^n;\Z^2)\cong  H^n(S^n\times S^n;\Z)\oplus H^n(S^n\times S^n;\Z)$.

We write $(0,1)$ for the element of $\Hom(\Z^2,\Z)$ sending $(x,y)$ to $y$, and likewise write $(1,0)$ for the element of $\Hom(\Z^2,\Z)$ sending $(x,y)$ to $x$, and set $\iota=(1,0)\oplus (0,1)\in H^n(K(\pi_n(X),n);\pi_n(X))\cong \Hom(\Z^2,\Z^2)\cong \Hom(\Z^2,\Z)\oplus \Hom(\Z^2,\Z)$. Then we have that $\phi_n^*(\iota)=\phi_n^*((1,0)\oplus(0,1))=(\alpha,0)\oplus (0,\gamma)\in H^n(S^n\times S^n;\Z^2)\cong H^n(S^n\times S^n;\Z)\oplus H^n(S^n\times S^n;\Z)$. Here $\alpha$ and $\gamma$ are preferred generators for $H^n(S^n\times S^n;\Z)\cong \Z^2$. It can now be seen that $k^{n+1}=\bar\tau((\alpha,0)\oplus(0,\gamma))$.

It suffices to examine the transgression from the Serre spectral sequence with integral coefficients for the Borel fibration in order to find out information about the first nontrivial $k$-invariant, $k^{n+1}$. In particular,
for $S^n\times S^n \to X\to BG $, we have:
\[E^{p,q}_2 = H^p(BG ; H^q(S^n \times S^n;\Z)) \Rightarrow H^{p+q}(X ; \Z).\]

The first nontrivial differential is $d_{n+1}$, and the transgression
\[
  d_{n+1} :H^0(BG;H^n(S^n\times S^n;\Z)) \to H^{n+1}( BG ; H^0(S^n\times S^n;\Z) ),
\]
here satisfies $d_{n+1}(\alpha)=\bar\tau(\alpha,0)$ and $d_{n+1}(\gamma)=\bar\tau(0,\gamma)$. It follows that $k^{n+1}=d_{n+1}(\alpha)\oplus d_{n+1}(\gamma)$. The cohomology ring of $H^*(BG;\Z)$ is given in Theorem \ref{BGintcohlgy} and we use the same notation by taking the generators to be $a$, $b$, and $c$, with $|a|=|b|=2$ and $|c|=3$, and $pa=pb=pc=c^2=0$. Additionally, we take $\alpha$ and $\gamma$ to be the generators in degree $n$ of $H^*(S^n\times S^n;\Z)$, with $\alpha^2=\gamma^2=0$, as described above. We see that the $E_2\cong E_{n+1}$ page reads:

\[
\xy
(0,50)*{};(0,0)*{} **\dir{-};
(0,0)*{};(125,0)*{} **\dir{-};
(-2,3)*{0};(-2,25)*{n};(-2,47)*{2n};
(4,-4)*{0};(18,-4)*{1};(32,-4)*{2};
(46,-4)*{3};(60,-4)*{4};(74,-4)*{5};
(88,-4)*{\dots};(102,-4)*{\dots};(116,-4)*{n+1};
(4,3)*{1};(18,3)*{0};(32,3)*{a,b};(46,3)*{c};
(60,3)*{b^2};(60,7)*{a^2,ab,};(74,3)*{bc};(74,7)*{ac};(88,3)*{};(102,3)*{};
(116,3)*{\dots, b^{(n+1)/2}};(116,7)*{a^{(n+1)/2}, \dots};
(4,25)*{\alpha,\gamma};(18,25)*{0};
(32,25)*{\alpha b, \gamma b};(32,29)*{\alpha a, \gamma a};
(46,25)*{\gamma c};(46,29)*{\alpha c};
(60,25)*{gens};(60,29)*{6};(74,25)*{\gamma ac, \gamma bc};(74,29)*{\alpha ac, \alpha bc};
(116,25)*{gens};(116,28)*{n+3};
{\ar^{d_{n+1}}(6,23)*{}; (108,9)*{}};
{\ar^{d_{n+1}}(6,44)*{}; (108,30)*{}};
(4,47)*{\alpha\gamma};(18,47)*{0};
(32,47)*{\alpha\gamma a};
(32,51)*{\alpha\gamma b};
(46,47)*{\alpha\gamma c};(60,47)*{gens};(60,51)*{3};
(74,47)*{\alpha\gamma bc};(74,51)*{\alpha\gamma ac};(116,47)*{gens};(116,51)*{(n+3)/2};
\endxy
\]

\noindent where $\alpha\gamma b$ is $\alpha\gamma\otimes b$, etc., by abuse of notation. Note that the blank entries are not necessarily 0.

By virtue of its codomain being generated by suitable powers of $a$ and
$b$, the transgression $d_{n+1}:E_{n+1}^{0,n}=H^0(BG;H^n(S^n\times S^n;\Z))\to E_{n+1}^{n+1,0}=H^{n+1}(BG;H^0(S^n\times S^n;\Z))$ satisfies
\begin{align*}
    d_{n+1}(\alpha)&=\sum_{i=0}^{(n+1)/2} q_{\alpha, i}a^{\frac{n+1}{2}-i}b^i \mbox{ and } \\
    d_{n+1}(\gamma)&=\sum_{j=0}^{(n+1)/2} q_{\gamma, j}a^{\frac{n+1}{2}-j}b^j,
\end{align*}
where the $q_{\alpha, i}$ and $q_{\gamma, j}$ are elements of $\Zp$.

This spectral sequence converges to the integral cohomology of $X$, and since $X$ is a finite manifold of dimension $2n$, there are restrictions on what the coefficients $q_{\alpha, i}$ and $q_{\gamma, j}$ can be.

\begin{proposition}\label{proposition:top-bottom-nonzero}
  The coefficients $q_{\alpha, 0}$ and $q_{\gamma, 0}$ (which are
  coefficients for $a^{(n+1)/2}$) cannot both be zero. Similarly, the
  coefficients $q_{\alpha,\frac{n+1}{2}}$ and
  $q_{\gamma, \frac{n+1}{2}}$ (which are coefficients for $b^{(n+1)/2}$)
  cannot both be zero.
\end{proposition}

\begin{proof}
  Since $G$ acts freely and $H^{2n}((S^n\times S^n)/G;\Z)\cong\Z$,
  only quotients of the groups generated by the $E^{p,q}_2\cong E^{p,q}_{n+1}$ terms with $p+q<2n$ or $p=0$ and   $q=2n$ can survive. Assume the transgression
  $d_{n+1}:E_{n+1}^{0,n}\to E_{n+1}^{n+1,0}$ satisfies
  $d_{n+1}(\alpha)=q_{\alpha, 1}a^{(n-1)/2}b+\dots +q_{\alpha, (n+1)/2}
  b^{(n+1)/2}$ and
  $d_{n+1}(\gamma)=q_{\gamma, 1}a^{(n-1)/2}b+\dots +q_{\gamma, (n+1)/2}
  b^{(n+1)/2}$, for some $q_{\alpha,i},q_{\gamma,j}\in\Zp$, $1\leq i,j\leq (n+1)/2$. In other words, both $q_{\alpha, 0}$ and $q_{\gamma, 0}$
  vanish.

  The $(n+1)^\st$ differential takes the generators in
  $E_{n+1}^{n-1,n}$ to combinations of the generators in
  $E_{n+1}^{2n,0}$.  By Leibniz $d_{n+1}$ sends
  $\alpha\otimes a^{(n-1)/2}$ to
  $q_{\alpha, 1}a^{n-1}b+\dots +q_{\alpha, (n+1)/2} a^{(n-1)/2}
  b^{(n+1)/2}$, and similarly for the other generators. It is not hard
  to see that the only other nontrivial differential, $d_{n+1}$, does not hit the subgroup generated by $a^{n+1}$, and there are no other differentials that map to this
  group. Therefore the generated $\Zp$ is present in $H^{2n}((S^n\times S^n)/G;\Z)$
  and other cohomology groups in higher degrees. Since
  $H^{2n}((S^n\times S^n)/G;\Z)$ is torsion free and the highest
  nontrivial degree, we get a contradiction.

  The argument for $q_{\alpha, \frac{n+1}{2}}$ and
  $q_{\gamma, \frac{n+1}{2}}$ both being nontrivial is similar.
\end{proof}

Observe that Proposition~\ref{proposition:top-bottom-nonzero} also implies that 
neither $d_{n+1}(\alpha)$ nor $d_{n+1}(\gamma)$ can map to $0$. We also see that it holds
after replacing the specified generators with their images under an
automorphism of $G$.

\begin{corollary}\label{corollary:restrictions_on_transgression}
  For nonzero $\lambda \in H^2(G;\Z)$, either $d_{n+1}(\alpha)$ or
   $d_{n+1}(\gamma)$ is nonzero in \[H^{n+1}(G;\Z) / \lambda^{(n+1)/2}.\]
\end{corollary}

\begin{proof}
  Suppose $\varphi$ is an automorphism of $G$ chosen so that
  $\varphi_\star \lambda = a \in H^2(G;\Z)$.  After twisting by
  $\varphi$ the action of $G$ on $S^n \times S^n$, the resulting
  quotient is homeomorphic (albeit not equivariantly homeomorphic) to
  the original quotient space.  In particular, in that quotient the
  coefficients $q_{\alpha, 0}$ and $q_{\gamma, 0}$, namely the
  coefficients for $a^{(n+1)/2}$, cannot both be zero, which
  corresponds in the original space to the condition in the
  Corollary.
\end{proof}

\section{Constructions}
\label{section:constructions}

Now we construct examples which are more complicated than lens spaces
cross lens spaces. In this section, we take the dimension of the spheres we are acting on to be $n=2m-1$ to avoid fractions appearing in subscripts. Let $R = (r_1,\ldots,r_m,r'_1,\ldots,r'_m)$ and
$Q = (q_1,\ldots,q_m,q'_1,\ldots,q'_m)$ be elements of $(\Zp)^{2m}$ so that
$R$ and $Q$ together generate a copy of $(\Zp)^2$ inside $(\Zp)^{2m}$.
We refer to these $4m$ parameters as ``rotation numbers'' in analogy
with the case of a lens space.

Let $S^{2m-1}$ be the unit sphere in $\C^m$, so $S^{2m-1} \times
S^{2m-1}$ is a submanifold of $\C^m \times \C^m$.  Then $R$ acts on
$S^{2m-1} \times S^{2m-1}$ by
\begin{align*}
R \cdot (z,z') &= (r,r') \cdot (z,z') \\
&= (r,r') \cdot (z_1,\ldots,z_m,z'_1,\ldots,z'_m) \\
&= \left( e^{2\pi i r_1/p} z_1, \ldots, e^{2\pi i r_1/p} z_m, e^{2\pi i r'_1/p} z'_1, \ldots e^{2\pi i r'_m/p} z'_m \right),
\end{align*}
and similarly $Q$ acts on $S^{2m-1} \times S^{2m-1}$.  This provides
an action of the group $(\Zp)^{2} \cong \langle R, Q \rangle$ on
$S^{2m-1} \times S^{2m-1}$.  In analogy with the lens space case, we call such actions ``linear'' and we
write the quotient as $L(p,p; R, Q)$.  In the case of lens spaces, the $k$-invariant is the product of rotation numbers.  We now compute the first nontrivial $k$-invariant in the case of $L(p,p; R, Q)$.  We will denote this first nontrivial $k$-invariant by $k$ in what follows.

\begin{lemma}
\label{lemma:product-of-rotation}
  Let $L = L(p,p;R,Q)$ and suppose $p > m$.  Then $k(L) \in H^{2m}((\Zp)^2; \Z^2)$ is 
  \[
    \left( \prod_{i=1}^m (r_i a + q_i b), \prod_{i=1}^m (r'_i a + q'_i b) \right),
  \]
  where $a$ and $b$ are generators of $H^2((\Zp)^2; \Z)$ as described in Section~\ref{section:cohomology}.
\end{lemma}

In keeping with the analogy to the lens space, Lemma~\ref{lemma:product-of-rotation} states that the $k$-invariant is the product of rotation \textit{classes} in $H^2((\Zp)^2; \Z)$.

\begin{proof}
  The $k$-invariant $k(L) \in H^{2m}(K(G,1);\Z^2)$ is a homotopy class of maps $K(\pi_1 L, 1) \to K(\pi_{2m-1} L, 2m)$.  The proof makes use of the naturality of the $k$-invariant.  Suppose
  a $\Zp$ subgroup of $(\Zp)^2$ is generated by $(\alpha,\beta)$.
  Then we have a cover
  \[
    \bar{L} = (S^{2m-1} \times S^{2m-1}) / \Zp \to L(p,p;R,Q).
  \]

  By \cite[page~396]{MR2721633}, the $k$-invariant
  $k(\bar{L}) \in H^{2m}(\Zp; \Z^2)$ associated to the quotient of
  $S^{2m-1} \times S^{2m-1}$ by the subgroup
  $\langle (\alpha,\beta) \rangle \cong \Zp$ is
  \[
    k(\bar{L}) = \left( \prod_{i=1}^m (r_i \alpha + q_i \beta) \omega, \prod_{i=1}^m (r'_i \alpha + q'_i \beta) \omega \right),
  \]
  where $\omega$ is the generator in $H^2(\Zp;\Z)$, which is
  identified with the generator of $\Zp$ via
  \(H^2(\Zp; \Z) \cong \Ext(H_1(\Zp;\Z), \Z) \cong \Zp\).

    By universal  coefficients and the fact that the cohomology (except in degree
  zero) of $\Zp$ and $(\Zp)^2$ is torsion, we have
  \(\Ext(H_1((\Zp)^2;\Z), \Z) \cong H^2((\Zp)^2; \Z) \cong (\Zp)^2\)
  and \(\Ext(H_1(\Zp;\Z), \Z) \cong H^2(\Zp; \Z) \cong \Zp\), and the
  map \(H^2((\Zp)^2; \Z) \to H^2(\Zp; \Z)\) is dual to the inclusion
  map $\Zp \hookrightarrow (\Zp)^2$; the inclusion map sends the
  generator of $\Zp$ to $(\alpha,\beta)$, so the dual map sends
  $xa + yb \in H^2((\Zp)^2; \Z)$ to $(\alpha a + \beta b) \omega$.
  
  By naturality of the $k$-invariant, we have that the map
  $H^{2m}((\Zp)^2; \Z^2) \to H^{2m}(\Zp; \Z^2)$ sends $k(L)$ to
  $k(\bar{L})$.  We consider only the left-hand factor of $k(L)$;
  this is some homogeneous polynomial of degree $n$ in the classes
  $a,b \in H^2((\Zp)^2; \Z)$; write this polynomial as $f(a,b)$.

  Then the map $H^{2m}((\Zp)^2; \Z) \to H^{2m}(\Zp; \Z)$ sends
  $f(a,b)$ to
  \[
    f(\alpha,\beta) \omega \in H^{2m}(\Zp; \Z)
  \]
  and
  therefore for $\alpha,\beta \in \Zp$,
  \[
    f(\alpha,\beta) = \prod_{i=1}^m (r_i \alpha + q_i \beta).
  \]
  Now assuming $m < p$, this equality of polynomials as functions
  gives rise to the desired equality
  \[
    f(a,b) = \prod_{i=1}^m (r_i a + q_i b).
  \]
  The right-hand factor of $k(L)$ is computed the same way.
\end{proof}

\section{The $S^3\times S^3$ classification}
\label{section:s3-times-s3}

Suppose $p > 3$.  We now classify $\Zp \times \Zp$ actions on
$S^3 \times S^3$ up to homotopy.  By Theorem~\ref{thm:he2}, this boils down to
the $k$-invariants encoded by the transgression
\begin{align*}
  d_4(\alpha)=q_{\alpha,0}a^2 + q_{\alpha,1} ab + q_{\alpha,2} b^2, \\
  d_4(\gamma)= q_{\gamma,0}a^2 + q_{\gamma,1} ab + q_{\gamma,2} b^2.
\end{align*}
We therefore package $(d_4(\alpha), d_4(\gamma))$ as a pair
$(Q_1,Q_2)$ of binary quadratic forms over $\Zp$.  The homotopy
classification of $(S^3 \times S^3)/(\Zp \times \Zp)$ amounts, algebraically, to
classifying pairs of binary quadratic forms over $\Zp$ up to the
action of automorphisms of $\Z^2$ on the pair $(Q_1,Q_2)$.  For
example, the pair $(Q_1,Q_2)$ determines the same equivariant oriented
homotopy type as $(Q_1 + Q_2, Q_1)$.  Note that $\Aut(\Z^2)$ amounts
to the action of
\[
\SL^{\pm}_2(\Zp) := \{ M \in \GL_2(\Zp) \mid \det M = \pm 1 \}.
\]
on pairs $(Q_1,Q_2)$.  In what follows, regard this as a \textit{left}
action of $\SL^{\pm}_2(\Zp)$, so that
$M = (m_{ij}) \in \SL^{\pm}_2(\Zp)$ acts via
\begin{equation}\label{eqn:left-action}\tag{*}
  M \cdot (Q_1,Q_2) = (m_{11} Q_1 + m_{12} Q_2, m_{21} Q_1 + m_{22} Q_2).
\end{equation}

Now we determine the classification disregarding the identification of
$\Zp \times \Zp$ with $\pi_1$.  On the levels of quadratic forms, we may
replace the pair $\left(Q_1,Q_2\right)$ by $\left(Q'_1,Q'_2\right)$
where $Q_1$ and $Q'_1$ (respectively $Q_2$ and $Q'_2$) are related by
a common change of coordinates, i.e., an automorphism of
$\Zp \times \Zp$ which amounts to $\GL_2(\Zp)$.  In what follows, regard
this as a \textit{right} action of $\GL_2(\Zp)$ on pairs
$\left(Q_1,Q_2\right)$.

\begin{lemma}\label{lemma:classification}
  Let $z$ be a quadratic nonresidue in $\Zp$.  A pair of binary
  quadratic forms $(Q_1,Q_2)$ satisfying the condition in
  Proposition~\ref{proposition:top-bottom-nonzero} is equivalent to
  $(xa^2, yb^2)$ or equivalent to $(a^2 + xb^2,2ab)$ for $x,y \in \Zp$.
\end{lemma}
\begin{proof}
There are five
\cite[Theorem IV.10]{MR0340283} equivalence classes of binary
quadratic forms modulo $p$, namely the trivial form $Q(a,b) = 0$, two
degenerate forms $a^2$ and $za^2$, and two nondegenerate quadratic
forms $a^2 + b^2$ and $a^2 + zb^2$.

Suppose $Q_1$ is degenerate, so $Q_1(a,b) = a^2$ or $Q_1(a,b) = za^2$.
Through an automorphism of $\Z^2$ subtracting a multiple of $Q_1$, the
form $Q_2$ becomes $x\, ab + y\, b^2$ for some $x, y \in \Zp$.  By
Proposition~\ref{proposition:top-bottom-nonzero}, it cannot be that
$y = 0$.  Since $y \neq 0$, the automorphism of $\Zp \times \Zp$ sending
$a \mapsto a$ and $b \mapsto \frac{-ax}{2y} + b$ preserves $Q_1$ but transforms
$Q_2$ into $y \, b^2 - \frac{x^{2}}{4 \, y} a^2$.  Subtracting off a
multiple of $Q_1$ via an automorphism of $\Z^2$ finally transforms
$Q_2$ into $y \, b^2$.  Therefore $(Q_1,Q_2) \simeq (xa^2, yb^2)$ for some
$x, y \in \Zp$.

On the other hand, suppose $Q_1$ is nondegenerate, meaning
$Q_1(a,b) = a^2 + b^2$ or $Q_1(a,b) = a^2 + zb^2$.  As before, by
subtracting off a multiple of $Q_1$, then the form $Q_2$ becomes
$x\, ab + y\, b^2$ for some $x, y \in \Zp$.  Either $y \neq 0$ or
$y = 0$.  If $y \neq 0$, then as before $Q_2$ is equivalent to
$y \, b^2 - \frac{x^{2}}{4 \, y} a^2$ which via an automorphism of
$\Z^2$ is transformed into a multiple of $b^2$, and this case is then
handled by the above case in which $Q_1$ is degenerate.  If $y = 0$,
then we are in the situation $(a^2+zb^2,xab)$ for nonzero $z$ and $x$.
We apply the automorphism $a \mapsto a$ and $b \mapsto 2b/x$ to reduce
to a situation of the form $(a^2 + wb^2,2ab)$ for some nonzero $w$.
\end{proof}

\begin{proposition}\label{proposition:xy-to-ab2ab}
  A pair of the form $(x \, a^2, y \, b^2)$ for nonzero $x, y \in \Zp$ is
  equivalent to $(a^2 + w \, b^2, 2 \, ab)$ for $w \in \Zp$.
\end{proposition}

\begin{proof}
  By independently scaling $a$ and $b$ and depending as to whether or
  not $x$ and $y$ are quadratic residues, the pair $(x a^2, y b^2)$ is
  equivalent to
  \[
    (a^2,b^2) \mbox{ or }
    (a^2,zb^2) \mbox{ or }
    (za^2,b^2) \mbox{ or }
    (za^2,zb^2)
  \]
  for a quadratic nonresidue $z \in \Zp^\star$.  By exchanging the
  roles of $a$ and $b$ and swapping the components of the tuple, the
  pair $(a^2,zb^2)$ is equivalent to $(za^2,b^2)$.  It is also the
  case that $(za^2,zb^2) \simeq (a^2,b^2)$ because
  \[
    (za^2,zb^2) \cdot \begin{pmatrix} 1 & 0 \\ 0 & 1/z \end{pmatrix}
    = \begin{pmatrix} z & 0 \\ 0 & 1/z \end{pmatrix} \cdot (a^2,b^2).
  \]

  To conclude the proof, we show
  $(a^2,w \, b^2) \simeq (a^2 + 4w^2 \, b^2, 2ab)$ for $w \in \Zp$.
  To see this,
  applying the equivalence
  $a \mapsto a/(2w) - b$ and $b \mapsto a + 2wb$ shows
  \[
    (a^2,w \, b^2) \simeq (\frac{1}{4 w^{2}} \, a^2 - \frac{1}{w} \, ab + b^{2}, w \, a^{2} + 4 w^2 \, a b + 4 w^3 \, b^{2} )
  \]
  and then applying the automorphism of $\Z^2$ corresponding to
  \[
    \begin{pmatrix} \frac{1}{4w^2} & \frac{-1}{2w} \\ w & 2w^2 \end{pmatrix} \in \SL_2(\Zp)
  \]
  implies that
  \[
    (a^2 + 4w^2 \, b^2, 2 \, ab) \simeq  (\frac{1}{4 w^{2}} \, a^2 - \frac{1}{w} \, ab + b^{2}, w \, a^{2} + 4 w^2 \, a b + 4 w^3 \, b^{2} )
  \]
  so $(a^2,w \, b^2) \simeq (a^2 + 4w^2 \, b^2, 2 \, ab)$.
\end{proof}

It remains to check that $\left(a^2 + b^2, 2ab\right)$ is \textit{not}
equivalent to $\left(a^2 + z b^2, 2ab\right)$.
\begin{lemma}\label{lemma:necessary}
  If $\left(a^2 + \delta b^2, 2ab\right)$ is
equivalent to $\left(a^2 + \delta' b^2, 2ab\right)$ for nonzero
$\delta$ and $\delta'$, then $\delta'/\delta \in \Zp^{\star 4}$
\end{lemma}
\begin{proof}
  We follow the argument in \cite{MR1684137}.  Suppose
  $\left(a^2 + \delta b^2, 2ab\right)$ is equivalent to
  $\left(a^2 + \delta' b^2, 2ab\right)$ for nonzero $\delta$ and
  $\delta'$.  Then there is an $R \in \GL_2(\Zp)$ and
  $S \in \pm \SL_2(\Zp)$ so that
  \begin{equation}\label{eq:equal-pairs}
    \left(a^2 + \delta b^2, 2ab\right) \cdot R  = S \cdot \left(a^2 + \delta' b^2, 2ab\right).
  \end{equation}
  Equality of the first component in each tuple yields
  \begin{equation}\label{eq:first-component}
    \left( {r_{11}}^2 + \delta {r_{21}}^2 \right) a^{2} + 2 \left(  r_{11} r_{12} + \delta r_{21} r_{22} \right) ab + \left( {r_{12}}^2 + \delta {r_{22}}^2 \right) b^2  =  s_{11} a^{2} + 2 s_{12}\, a b  + \delta' s_{11} \, b^{2} 
  \end{equation}
  Equality of the coefficients of $a^2$ and $b^2$ in \eqref{eq:first-component} yields
  \begin{align*}
    s_{11} &= \delta {r_{21}}^2 + {r_{11}}^2 \mbox{ and } \\
    \delta' s_{11} &= \delta {r_{22}}^2 + {r_{12}}^2,
  \end{align*}
  respectively, and therefore
  \begin{equation}\label{eq:e1}
    \delta {r_{22}}^2 + {r_{12}}^2 = \delta' \delta {r_{21}}^2 + \delta' {r_{11}}^2.
  \end{equation}
  Equality of the second component in \eqref{eq:equal-pairs} yields
  \begin{equation}\label{eq:second-component}
    2  r_{11} r_{21} \, a^{2}
    + \left( 2  r_{12} r_{21} + 2  r_{11} r_{22} \right)  a b
    + 2  r_{12} r_{22} \, b^{2}
    =   s_{21} a^2 + 2  s_{22} \, a b +  \delta' s_{21} b^{2}.
  \end{equation}
  Equality of the coefficients of $a^2$ and $b^2$ in \eqref{eq:second-component} yields
  \begin{align*}
    s_{21} &= 2 \, r_{11} r_{21} \mbox{ and }\\
    \delta' s_{21} &= 2 \, r_{12} r_{22},
  \end{align*}
  respectively.  We conclude
  \begin{equation}\label{eq:e2}
    r_{12} r_{22} = \delta' r_{11} r_{21}.
  \end{equation}
  Squaring both sides of \eqref{eq:e1} and subtracting $4\delta$ times \eqref{eq:e2}  squared yields
  \[
    {\left(\delta {r_{22}}^2 - {r_{12}}^2\right)}^{2} = {\left(\delta' \delta {r_{21}}^2 - \delta' {r_{11}}^2\right)}^{2}
  \]
  and so
  \begin{equation}\label{eq:two-cases}
    \delta {r_{22}}^2 - {r_{12}}^2 = \pm {\left(\delta' \delta {r_{21}}^2 - \delta' {r_{11}}^2\right)}.
  \end{equation}
  The sign in \eqref{eq:two-cases} cannot be positive; if it were, then adding \eqref{eq:two-cases} to \eqref{eq:e1} yields
  \[
    2 \, \delta {r_{22}}^2 = 2 \, \delta' \delta {r_{21}}^2,
  \]
  so ${r_{22}}^2 = \delta' {r_{21}}^2$.  But multiply both sides of
  \eqref{eq:e2} by ${r_{21}}^2$ and we deduce
  \[
    r_{12} r_{22} {r_{21}}^2 = \delta' r_{11} {r_{21}}^2 = {r_{22}}^2 r_{11}.
  \]
  So either $r_{22} = 0$, in which case $r_{21} = 0$ and the second row of $R$ is zero, or 
  $r_{12} r_{21} = r_{22} r_{11}$ and so $\det R = 0$.  In either case we contradict the assumption $R \in \GL_2(\Zp)$ and so the sign in \eqref{eq:two-cases} must be negative meaning
  \begin{equation}\label{eq:negative-sign}
    \delta {r_{22}}^2 - {r_{12}}^2 = - \delta' \delta {r_{21}}^2 + \delta' {r_{11}}^2.
  \end{equation}
  The difference of \eqref{eq:e1} and \eqref{eq:negative-sign} yields
  \[
    2 {r_{12}}^2 = 2 \, \delta \delta' {r_{21}}^2
  \]
  so $\delta \delta'$ is a square in $\Zp$.  And if our only
  requirement is that $R \in \GL_2(\Zp)$, then the necessary condition
  that $\delta' \delta \in \Zp^{\star 2}$ would also suffice, but we
  also required $S \in \SL^{\pm}tf_2(\Zp)$, or equivalently that
  $(\det S)^2 = 1$.  From \eqref{eq:first-component} and
  \eqref{eq:second-component}, we have
  \[
    S = \begin{pmatrix} s_{11} & s_{12} \\ s_{21} & s_{22} \end{pmatrix} =
    \begin{pmatrix} \delta \, {r_{21}}^2 + {r_{11}}^2 & \delta \, r_{21} \, r_{22} + r_{11} \, r_{12} \\
      2 \, r_{11} \, r_{21} & r_{12} \, r_{21} + r_{11} \, r_{22} \end{pmatrix},
  \]
  which means
  \[
    \det S = {\left(\delta {r_{21}}^2 - {r_{11}}^2\right)} {\left(r_{12} r_{21} - r_{11} r_{22}\right)}.
  \]
  Squaring $\det R$ results in
  \begin{align*}
    (\det R)^2 &= (r_{11} \, r_{22} - r_{12} \, r_{21})^2  \\
               &= {r_{11}}^2 \, {r_{22}}^2 - 2 \, r_{11} \, r_{12} \, r_{21} \, r_{22} + {r_{12}}^2 \, {r_{21}}^2 \\
               &= {r_{11}}^{4} \, \frac{\delta'}{\delta} - 2 \, \delta' \, {r_{11}}^2 \, {r_{21}}^2 + \delta \, \delta' \, {r_{21}}^{4} \\
               &= \frac{\delta'}{\delta} \, \left( \delta \, {r_{21}}^2 - {r_{11}}^2 \right)^2.
  \end{align*}
  by invoking Equation~\eqref{eq:e2}
  and applying the identities
  ${r_{12}}^2 = \delta \, \delta' \, {r_{21}}^2$ and $\delta \, {r_{22}}^2 = \delta' \, {r_{11}}^2$, which follow from taking the sum and difference of Equations~\eqref{eq:e1} and~\eqref{eq:negative-sign}.
  Consequently 
  \begin{align*}
    (\det S)^2 &= {\left(\delta {r_{21}}^2 - {r_{11}}^2\right)^2} \left( \det R \right)^2 \\
               &= \frac{\delta'}{\delta} \left(\delta {r_{21}}^2 - {r_{11}}^2\right)^4,
  \end{align*}
  so $\delta'/\delta$ must be a fourth power.
\end{proof}
In particular, $\left(a^2 + b^2, 2ab\right)$ is not equivalent to
$\left(a^2 + z b^2, 2ab\right)$ because $z$ was chosen specifically to
be a quadratic nonresidue.

\begin{lemma}\label{lemma:construct-equivalence}
  For nonzero $\delta, w \in \Zp$, the pair $\left(a^2 + \delta b^2, 2ab\right)$ is equivalent to $\left(a^2 + \delta w^4 \, b^2, 2ab\right)$.
\end{lemma}
\begin{proof}
  Choose $r_1,r_2 \in \Zp$ so that
  \begin{equation}\label{eqn:curve}
    \delta {r_{1}}^2 - r_{2}^2 \equiv 1/w^3 \pmod p.
  \end{equation}
  This is possible; in fact, there are $p - \legendre{\delta}{p}$ solutions to \eqref{eqn:curve}.
  Then set
  \[
    R := \begin{pmatrix}
w^{2} r_{2}  & \delta w^{2} r_{1} \\
r_{1} & r_{2}
\end{pmatrix}
          \mbox{ and }
S := \begin{pmatrix}
\delta w^{4} {r_{1}}^{2}  + w^{4} r_{2}^{2}  & 2 \delta w^{4} r_{1} r_{2}  \\
2  w^{2} r_{1} r_{2}  & \delta w^{2} r_{1}^{2}  + w^{2} r_{2}^{2} 
\end{pmatrix}.
\]
Because of \eqref{eqn:curve}, we have
\begin{align*}
\det R &= - w^2 \left( \delta {r_{1}}^{2} - {r_{2}}^{2} \right) = -1/w \neq 0, \\
\det S &= w^6 {\left(\delta {r_{1}}^{2} - {r_{2}}^{2}\right)}^{2} = 1,
\end{align*}
so $R \in \GL_2(\Zp)$ and $S \in \SL_2(\Zp)$.

We finish the proof by verifying
\begin{equation}\label{eqn:goal}
  \left(a^2 + \delta w^4 b^2, 2ab\right) \cdot R  = S \cdot \left(a^2 + \delta b^2, 2ab\right).
\end{equation}
Comparing the first coordinates each side of~\eqref{eqn:goal} shows
\[
  {\left(r_{2} w^{2} a + \delta r_{1} w^{2} b \right)}^{2} + \delta w^{4} {\left(r_{1} a + r_{2} b\right)}^{2}
  =
  {\left(\delta w^{4} {r_{1}}^{2}  + w^4 {r_{2}}^{2} \right)} \cdot {\left(a^{2} + \delta b^{2}\right)} + 2 \delta w^{4} r_{1} r_{2}  \cdot 2 a b.  
\]
Similarly the second coordinates are equal because
\[
  2 \, {\left(w^2 r_{2} a + \delta w^2 r_{1} b\right)} {\left(r_{1} a + r_{2} b\right)}
  =
  2 w^2  r_{1} r_{2}  {\left(a^2 + b^{2} \delta\right)}  +  {\left(\delta w^2 {r_{1}}^{2} + w^2 {r_{2}}^{2} \right)}\, 2 a b.
\]
\end{proof}
It is easier to see that $\left(a^2 + \delta b^2, 2ab\right)$ is
equivalent to $\left(a^2 + \delta w^8 \, b^2, 2ab\right)$.  Simply
replace $a$ by $aw$ and $b$ by $b/w^3$ to show
$(a^2 + \delta \, w^8 b^2, 2 \, ab) \simeq (w^2 a^2 + \delta w^2 \,
b^2, \frac{2}{w^2} \, ab)$, and then scale the first by $1/w^2$ and
the second by $w^2$ to see this is equivalent to
$(a^2 + \delta \, b^2, 2 \, ab)$.  The challenge of
Lemma~\ref{lemma:construct-equivalence} lies in replacing
$w^8$ with $w^4$.

Combining Lemmas~\ref{lemma:necessary}
and~\ref{lemma:construct-equivalence} yields the following.
\begin{proposition}\label{proposition:classification-of-forms}
  Equivalence classes of pairs of the form $(a^2 + w b^2, 2ab)$ are in
  one-to-one correspondence with elements of
  $\Zp^\times / (\Zp^\times)^4$ where $\Zp^\times$ denotes units modulo $p$.
\end{proposition}

Observe that the size of $\Zp^\star / (\Zp^{\times})^4$ depends on
$p \mod 4$. Specifically, for $p \equiv 1 \bmod 4$, there are \textit{four}
equivalence classes.  These are given by $(a^2 + z b^2, 2ab)$ for $z$
representatives of classes $\Zp^\star / \Zp^{\star 4}$.

For $p \equiv 3 \bmod 4$, there are \textit{two} equivalence classes.
For nonzero $x, x', y, y' \in \Zp$, the pair $(x a^2, y b^2)$ is
equivalent to $(x' a^2, y' b^2)$, and every pair is equivalent to
either $(a^2+b^2,2ab)$ or $(a^2 + zb^2, 2ab)$ for a quadratic
nonresidue $z$.  So the only possibilities are
$(a^2 + b^2, 2ab) \simeq (a^2,b^2)$ and $(a^2 + z b^2, 2ab)$.

All of this algebra encodes the homotopy type of the quotients, as
summarized in the following.
\begin{theorem}\label{thm:summary}
Let $p > 3$ be prime. If $p \equiv 1 \bmod 4$, then there are four homotopy classes of quotients of $S^3 \times S^3$ by free $\Zp \times \Zp$ actions.  If $p \equiv 3 \bmod 4$, then there are two classes.
\end{theorem}
\begin{proof}
  We must construct quotients of $S^3 \times S^3$ by free
  $\Zp \times \Zp$ actions which exhibit these possible
  $k$-invariants.  For this, we rely on
  Lemma~\ref{lemma:product-of-rotation}.  We note that
  $(a^2 + w b^2, 2ab)$ is equivalent to
  \[
    \left( a^2 + wb^2 + (1+w)\,ab, 2ab \right) =
    \left( \left(a+b\right) \left(a+w\,b\right), 2 ab \right),
  \]
  so let $R = (1, 1, 2, 0)$ and $Q = (1, w, 0, 1)$ and then
  $L(p,p;R,Q)$ has $k$-invariant equivalent to $(a^2 + w b^2, 2ab)$.
  We must impose the additional condition $w \neq 0$ in order to ensure
  that this is a \textit{free} action.
  With this constructions in hand, the classification of quotients
  then follows from
  Proposition~\ref{proposition:classification-of-forms}.
\end{proof}

\begin{remark}
  There are precedents for considering the simultaneous equivalence of
  forms.  The case of simultaneous equivalence of forms over $\Z$ is
  discussed in \cite{MR1133952} but our situation over $\Zp$ is
  easier.  To make the situation even more concrete, instead of forms,
  consider matrices; equivalence of forms amounts to congruence of
  matrices.  That setup fits into the work of Corbas and Williams
  \cite{MR1684137} which considers the action of
  $\GL_2(\Zp) \times \GL_2(\Zp)$ on pairs $\left(A,B\right)$ of
  matrices, where $\GL_2$ acts on the right by congruence and on the
  left as in \eqref{eqn:left-action}.
\end{remark}

\section{Lens cross lens}
\label{section:lens-times-lens}

Section~\ref{section:s3-times-s3} completed the classification of
$\Zp \times \Zp$ actions on $S^3 \times S^3$, but now we narrow in on a special
case.  Consider the case of $L_3(p;1,x) \times L_3(p;1,y)$, i.e., the
product of two lens spaces with rotation numbers $x$ and $y$
respectively.  Viewed as a quotient of $S^3 \times S^3$ by
$\Zp\times\Zp$, this product of lens spaces has $k$-invariant
$(xa^2,yb^2)$.

We can classify $L_3(p;1,x) \times L_3(p;1,y)$ up to (simple) homotopy
equivalence.  When $p \equiv 3 \bmod 4$, any product of three-dimensional
lens spaces is (simple) homotopy equivalent to any other such product.

\begin{proposition}\label{proposition:lens-cross-lens}
  Suppose $p \equiv 3 \bmod 4$.  Then for nonzero
  $x, x', y, y' \in \Zp$, the pair $(x a^2, y b^2)$ is equivalent to
  $(x' a^2, y' b^2)$.
\end{proposition}
\begin{proof}
  As in the proof of Proposition~\ref{proposition:xy-to-ab2ab}, we have
  the pair $(x a^2, y b^2)$ is
  equivalent to
  \[
    (a^2,b^2) \simeq (za^2,zb^2) \mbox{ or }
    (a^2,zb^2) \simeq
    (za^2,b^2).
  \]
  for a quadratic nonresidue $z \in \Zp^\star$.  But when
  $p \equiv 3 \bmod 4$, the quantity $-z$ is a square, and so
  \[
    (a^2,zb^2) \simeq (a^2,-zb^2) \simeq (a^2,b^2)
  \]
  meaning \textit{all} pairs of the form $(xa^2,yb^2)$ are equivalent.
\end{proof}

When $p \equiv 1 \bmod 4$, since
\[
  (xa^2,yb^2) \simeq (a^2, (y/x) \, b^2) \simeq (a^2 + 4\,(y/x)^2 \, b^2, 2 \, ab),
\]
the classification boils down to whether or not $2(y/x)$ is a square modulo $p$. 

This is related to previous work of Kwasik--Schultz; they proved
squares of lens spaces are diffeomorphic.
\begin{theorem}[\cite{MR1876893}]
  For $p$ odd and rotation numbers $r$ and $q$, there is a
  diffeomorphism
  \[
    L_3(p;1,r) \times L_3(p;1,r) \cong L_3(p;1,q) \times L_3(p;1,q).
  \]
\end{theorem}

A future paper completes the homeomorphism classification of spaces
resulting from ``linear'' actions such as these products of lens
spaces.

\section{Some comments on groups containing $\Zp\times\Zp$}
\label{section:NormalAbelianSubgroup}

While we know that $\Zp$ and $\Zp\times\Zp$ can act freely on $S^n\times S^n$, the exact conditions for a group to be able to act freely on $S^n \times S^n$ remains open. Conner \cite{MR0096235} and Heller \cite{MR0107866} showed that for a group to act freely on $S^n\times S^n$, the group must have rank at most two, but Oliver \cite{MR561237} showed that $A_4$ cannot act on $S^n\times S^n$, and so every rank 2 simple group is also ruled out \cite{MR1745517}. Explicit examples of free actions by subgroups of a non-abelian extension of $S^1$ by $\Zp\times\Zp$ have been constructed \cite{MR2592957}, but Okay--Yal\c{c}in \cite{okay2018} have shown that $\Qd(p)=(\Zp\times\Zp)\rtimes SL_2(\mathbb{F}_p)$ cannot act freely on $S^n\times S^n$. In this section we show how the restrictions on the $k$-invariant as described in Section \ref{section:k-invariants} can be useful in determining whether or not a group $G$ containing $\Zp\times\Zp$ as a normal abelian subgroup can act freely on $X=S^n\times S^n$. We continue to take $p > 3$ to be an odd prime and $n\ge 3$ to be odd. We align some of our notation with that in \cite{okay2018} to better show the parallel calculations.

Similar to the approach in Section~\ref{section:k-invariants}, we can consider the Borel fibration:
\[X \overset{i}\to X_{hG}\to BG,\]
and the associated Serre spectral sequence
\[E^{p,q}_2 = H^p(BG ; H^q(X;\Z)) \Rightarrow H^{p+q}(X_{hG} ; \Z).\]
with the first nontrivial differential $d_{n+1}$.
If $\alpha$ and $\gamma$ are the generators in degree $n$ of $H^*(X;\Z)$, with $\alpha^2=\gamma^2=0$, then $d_{n+1}(\alpha)=\bar\tau(\alpha,0)$, $d_{n+1}(\gamma)=\bar\tau(0,\gamma)$, and $k^{n+1}=d_{n+1}(\alpha)\oplus d_{n+1}(\gamma)$.

Set $K$ to be the normal Abelian subgroup of $\Zp\times \Zp$ in $G$, and consider the restriction of the spectral sequence associated to the Borel fibration to the $K$ action. Then Proposition \ref{proposition:top-bottom-nonzero} and Corollary \ref{corollary:restrictions_on_transgression} can be sometimes be used to determine if $G$ can act freely on $X$.
The transgression for the first nontrivial differential of the restriction of the spectral sequence associated to Borel fibration to $K$ is
\[(d_{n+1})_K : H^0(BK;H^n(X;\Z)) \to H^{n+1}( BK ; H^0(X;\Z) ).
\]
Let $\Res^G_K:H^*(G)\to H^*(K)$ be induced by the inclusion of $K$ into $G$. Since the Borel construction is natural, it follows that the $k$-invariant in the restricted case is $k^{n+1}_K=\Res^G_K(d_{n+1}(\alpha))\oplus \Res^G_K(d_{n+1}(\gamma))$.

Suppose $G$ acts freely on $X$, so $H^*(X_{hG};\Z)\cong H^*(X/G;\Z)$ is finite-dimensional in each degree and vanishes above $2n$. It follows that the restriction to $K$ gives $H^*(X_{hK};\Z)\cong H^*(X/K;\Z)$ is also finite dimensional in each degree and vanishes above $2n$ as $K$ acts freely. If both $(d_{n+1})_K(\alpha)$ and $(d_{n+1})_K(\gamma)$ are zero in $H^{n+1}(K;\Z)/\lambda^{(n+1)/2}$, for some nonzero $\lambda\in H^2(K;\Z)$, then $X/K$ will fail to be finite dimensional by Corollary \ref{corollary:restrictions_on_transgression}, and we get a contradiction. Hence $G$ cannot act freely. 

As an example, consider $G=\Qd(p)=(\Zp)^2\rtimes SL_2(\Zp)$. We show that one can use the restrictions on the $k$-invariants and some of the arguments in \cite{okay2018} to see that $\Qd(p)$ cannot act freely on $S^n\times S^n$ for $p$ an odd prime and $n$ odd. This result is consistent with Theorem~5.1 in \cite{okay2018}. 

Since cohomology is taken with $\Zp$ coefficients in \cite{okay2018}, we first set up a relationship between generators with from the different coefficient groups. Suppose the first nontrivial differential takes $\alpha$ and $\gamma$, also the generators of $H^n(S^n\times S^n;\Zp)$ by slight abuse of notation, to $\mu_1$ and $\mu_2$ in $H^{n+1}(G;\Zp)$. Taking $K$ to be the normal elementary Abelian subgroup $\Zp\times \Zp$ in $G=\Qd(p)$, and restricting the action to $K$, we have that $\theta_1,\theta_2 \in H^{n+1}(K;\Zp)$ are such that $\theta_1=\Res_K^G(\mu_1)$ and $\theta_2=\Res_K^G(\mu_2)$.

Recall the commuting triangle from Section \ref{section:cohomology}:

\[\xymatrix{H^n(K;\Z)\ar[r]^{\rho} & H^n(K;\Zp)\ar[r]^{\tilde{\beta}}\ar[rd]_{\beta} & H^{n+1}(K;\Z)\ar[r]^p\ar[d]^{\rho} & H^{n+1}(K;\Z)\\
 & & H^{n+1}(K;\Zp) & }\]

Since $p$ is the $0$ map, the vertical $\rho$ is injective and $\tilde{\beta}$ is surjective. We can write $H^*(K;\Z/p)=\mathbb{F}_p[x,y]\otimes\wedge(u,v)$, where $|x|=|y|=2$, $|u|=|v|=1$, and $\beta(u)=x$, $\beta(v)=y$, and $H^*(K;\Z)=\mathbb{F}_p[a,b]\otimes \wedge(c)$, with $|a|=|b|=2$, $|c|=3$. It is not hard to see that $\tilde{\beta}(x)=a$, $\tilde{\beta}(y)=b$, and $\tilde{\beta}(uv)=c$.

Now the Bockstein generally satisfies $\beta(\delta\varepsilon)=\beta(\delta)\varepsilon+(-1)^{|\delta|}\delta\beta(\varepsilon)=\delta\beta(\varepsilon)$, for $\delta$ being $x^iy^j$ and $\varepsilon$ being $u$, $v$, or $uv$. We see that
\[
  \beta(H^n(K;\Zp))\subseteq \langle x^{(n+1)/2},x^{(n-1)/2}y,\dots,y^{(n+1)/2}\rangle\subseteq \mathbb{F}_p[x,y],
\]
since $n+1$ is even. Similarly, $\tilde{\beta}$ satisfies $\tilde{\beta}(\delta\varepsilon)=\delta\tilde{\beta}(\varepsilon)$, for $\delta$ being $x^iy^j$ and $\varepsilon$ being $u$, $v$, or $uv$. Again we see that $\tilde{\beta}(H^n(K;\Zp))\subseteq \langle a^{(n+1)/2},a^{(n-1)/2}b,\dots,b^{(n+1)/2}\rangle\subseteq \mathbb{F}_p[a,b]$. As $\tilde{\beta}$ is surjective, $\rho$ is injective, and $\beta=\rho(\tilde{\beta})$, it follows that the $k$-invariant $\theta_1\oplus\theta_2$ comes from elements in $H^{n+1}(K;\Z)$ for some $K$ action on $S^n\times S^n$: $k^{n+1}=\rho^{-1}(\theta_1)\oplus\rho^{-1}(\theta_2)$.

In \cite{okay2018}, it is shown that the ideal generated by $\theta_1$ and $\theta_2$ is in fact generated by $\zeta^{(n+1)/2(p+1)}$, where $\zeta=xy^p-yx^p$ (which is in part based on calculations in \cite{MR2712167}). Since no power of $\zeta$ will contain $x^{(n+1)/2}$ or $y^{(n+1)/2}$, we see that $d_{n+1}(\alpha)$ and $d_{n+1}(\gamma)$, where $\alpha$ and $\gamma$ generate $H^n(S^n\times S^n;\Z)$, have both $q_{\alpha, 0}$ and $q_{\gamma, 0}$ are zero (where $q_{\alpha, 0}$ and $q_{\gamma, 0}$ are the coefficients in Proposition~\ref{proposition:top-bottom-nonzero}), we derive a contradiction.

It is worth noting that in \cite{okay2018}, the calculations show that the free actions of $\Qd(p)$ must have $p$ smaller than $n$, and $n+1$ divisible by $2(p+1)$. The argument also finds a contradiction to finiteness, but relies on \cite{Carlsson}. We also note that while we take $p$ to be large in our homotopy type calculations, the only restrictions that were required in Section~\ref{section:k-invariants} (and hence in this section) were that $p>3$ be an odd prime and $n\geq 3$ be odd. Further, there may be a way to show a contradiction to finiteness using Proposition~\ref{proposition:top-bottom-nonzero} more directly (without needing to make arguments with $\Zp$ coefficients.)

A similar argument could hold for any group containing $(\Zp)^2$ that has a restriction that forces the transgression to behave in such a way.

\bibliographystyle{plain}
\bibliography{references.bib}

\end{document}